\documentclass[a4j,12pt,showkeys]{article}

\usepackage{amsmath,amsthm,amssymb}
\usepackage{geometry}
\usepackage{hyperref}
\usepackage{cancel}
\usepackage{mathtools}
\usepackage{bm}

\theoremstyle{definition}
\date{}

\newtheorem{theo}{Theorem}[section]
\newtheorem{defi}[theo]{Definition}
\newtheorem{lemm}[theo]{Lemma}
\newtheorem{prop}[theo]{Proposition}

\newtheorem{rem}[theo]{Remark}
\newtheorem*{assum}{Assumption}

\newcommand{\relmiddle}[1]{\mathrel{}\middle#1\mathrel{}}
 
\newcommand{\1}{\mbox{1}\hspace{-0.25em}\mbox{l}}

\newcommand{\esssup}{\underset{s\in[0,T]}{\text{ess\,sup}\,}}
\providecommand{\keywords}[1]{\textbf{Keywords:} #1}

\allowdisplaybreaks[0]

\makeatletter

\@addtoreset{equation}{section}
\makeatother
\def\widebar{\accentset{{\cc@style\underline{\mskip10mu}}}}
\numberwithin{equation}{section}
\def\rnum#1{\expandafter{\romannumeral #1}} 
\def\Rnum#1{\uppercase\expandafter{\romannumeral #1}}

\title{Small-time solvability of a flow of forward-backward stochastic differential equations}

\author{Yushi Hamaguchi\thanks{Department of Mathematics, Kyoto University, Kyoto 606--8502, Japan, \href{mailto:hamaguchi@math.kyoto-u.ac.jp}{hamaguchi@math.kyoto-u.ac.jp}}}

\begin{document}
\maketitle


\begin{abstract}
Motivated from time-inconsistent stochastic control problems, we introduce a new type of coupled forward-backward stochastic systems, namely, flows of forward-backward stochastic differential equations. They are coupled systems consisting of a single forward SDE and a continuum of BSDEs, which are defined on different time-intervals. We formulate a notion of equilibrium solutions in a general framework and prove small-time well-posedness of the equations. We also consider discretized flows and show that their equilibrium solutions approximate the original one, together with an estimate of the convergence rate.
\end{abstract}



\section{Introduction}

\paragraph{}
\ \,In this paper, we introduce a new type of coupled forward-backward stochastic systems, namely, flows of forward-backward stochastic differential equations. They are coupled systems consisting of a single forward stochastic differential equation (SDE) and a continuum of backward stochastic differential equations (BSDEs), which are defined on different time-intervals. The solution, which we call the equilibrium solution, of a flow of forward-backward SDEs consists of a family of processes $\left(X,\{Y^t,Z^t\}_{t\in[0,T]}\right)$ where $X$ is an adapted process defined on an interval $[0,T]$, $\left(Y^t,Z^t\right)$ is a pair of adapted processes defined on $[t,T]$ for each $t\in[0,T]$, and they satisfy the following system in the It\^{o} sense:
\begin{equation}\label{flow of FBSDE}
\begin{cases}
	dX_s=B(s,X_s,Y^s_s)\,ds+\Sigma(s,X_s,Y^s_s)\,dW_s,\ s\in[0,T],\\
	dY^t_s=-F(t,s,X_t,\mathbb{E}_t[X_s],X_s,Y^s_s,Y^t_s,Z^t_s)\,ds+Z^t_s\,dW_s,\ s\in[t,T],\\
	X_0=x,\ Y^t_T=G(t,X_t,\mathbb{E}_t[X_T],X_T),\ t\in[0,T].
\end{cases}
\end{equation}
Here, $\mathbb{E}_t[\cdot]$ denotes the conditional expectation given $\mathcal{F}_t$; $B$, $\Sigma$, $F$ and $G$ are given random functions and $x$ is a given initial condition for $X$. The first line of (\ref{flow of FBSDE}) is an SDE with the initial condition $X_0=x$ which determines the time evolution of the process $X$ and, for each $t\in[0,T]$, the second line is a BSDE defined on $[t,T]$ with the terminal condition $Y^t_T=G(t,X_t,\mathbb{E}_t[X_T],X_T)$ which determines the pair of processes $\left(Y^t,Z^t\right)$. All the above equations are coupled via the ``diagonal term'' $Y^s_s$ of the random field $Y^t_s;\ 0\leq t\leq s\leq T$. The arguments of $F$ and $G$ in the BSDEs depend on the current time $t$, the current state $X_t$, and the conditional expectation $\mathbb{E}_t[X_s]$ or $\mathbb{E}_t[X_T]$. Thus, the system (\ref{flow of FBSDE}) is regarded as a generalization of classical forward-backward SDEs to a time-inconsistent setting.
\par
This type of systems of equations appear in time-inconsistent stochastic control problems and characterize their (open-loop) equilibrium controls. Time-inconsistent control problems are recently studied by Ekeland and Lazrak~\cite{a_Ekeland-Lazrak_10}, Yong~\cite{a_Yong_12}, Bj\"{o}rk, Murgoci and Zhou~\cite{a_Bjork-_14}, Hu, Jin and Zhou~\cite{a_Hu-12,a_Hu-17}, Djehiche and Huang~\cite{a_Djehiche-Huang_16}, Bj\"{o}rk, Khapko and Murgoci~\cite{a_Bjork-_17}, Wei, Yong and Yu~\cite{a_Wei_17}, Yong~\cite{a_Yong_17}, Yan and Yong~\cite{I_Yan-Yong_19}, among others. Time-inconsistency occurs for example when the player's time preference is described by a non-exponential discount function such as the hyperbolic discounting, or when the cost functional is a nonlinear function of (conditional) expectation of a state process such as dynamic mean-variance control problems. Unlike classical control problems the so-called Bellman's principle does not hold in these cases. In other words, a strategy which is optimal at a given starting point is no longer optimal when viewed from a later date and different state. Thus, we have to reconsider the concept of ``optimality''. Hu, Jin and Zhou~\cite{a_Hu-12,a_Hu-17} introduced an alternative concept of optimality in time-inconsistent problems, namely, an (open-loop) equilibrium control, which is time-consistent. We overview its definition and connection to flows of forward-backward SDEs by informal arguments below. For a detailed discussion of equilibrium controls for time-inconsistent stochastic control problems, see \cite{a_Hu-12,a_Hu-17,I_Yan-Yong_19,a_Yong_17}.
\par
Let $\mathcal{U}[t,T]$ be a set of control processes on $[t,T]$ taking values in a Borel subset $U$ of a Euclidean space. For each $u\in\mathcal{U}[t,T]$, let $x^u$ be the corresponding controlled state process defined as the unique solution of the SDE
\begin{equation*}
\begin{cases}
	dx^u_s=b(s,x^u_s,u_s)\,ds+\sigma(s,x^u_s)\,dW_s,\ s\in[t,T],\\
	x^u_t=x_t,
\end{cases}
\end{equation*}
where $(t,x_t)\in[0,T]\times\mathbb{R}$ is an initial pair. Define the player's cost functional that is viewed at the initial pair $(t,x_t)$ by
\begin{equation}\label{intro cost}
	J(t,x_t;u):=\mathbb{E}_t\left[\int^T_tf(t,x_t,s,x^u_s,u_s)\,ds+g(t,x_t,x^u_T)\right]+h(t,x^u_t,\mathbb{E}_t[x^u_T]).
\end{equation}
Here, we assume that all given functions $b,\ \sigma,\ f,\ g,\ h$ are deterministic, one-dimensional and sufficiently smooth for simplicity. In the following sections, we consider multi-dimensional and random coefficients.
\par
The player's objective is to search for an ``optimal'' strategy through the time-interval $[t,T]$. This problem is time-inconsistent since (\rnum{1}) the cost functional depends on the current time and state $(t,x_t)$, and (\rnum{2}) the second term of the right hand side of (\ref{intro cost}) is a (nonlinear) function of the conditional expectation of the terminal state. For an initial state $x_0\in\mathbb{R}$ at time $0$, we call $\hat{u}\in\mathcal{U}[0,T]$ an (open-loop) equilibrium control with respect to $x_0$ and $\hat{x}:=x^{\hat{u}}$ the corresponding equilibrium state process if it satisfies
\begin{equation*}
	\liminf_{\epsilon\downarrow0}\frac{J(t,\hat{x}_t;u^{t,\epsilon,v})-J(t,\hat{x}_t;\hat{u})}{\epsilon}\geq0\ \text{a.s.}
\end{equation*}
for any $t\in[0,T)$ and any square-integrable $\mathcal{F}_t$-measurable random variable $v$ taking values in $U$, where $u^{t,\epsilon,v}$ is the ``spike variation'' of $\hat{u}$ at time $t$ with respect to $v$, namely, $u^{t,\epsilon,v}_s:=v$ if $s\in[t,t+\epsilon)$ and $u^{t,\epsilon,v}_s:=\hat{u}_s$ otherwise. Then, by a version of the stochastic maximum principle (see \cite{I_Yan-Yong_19}), an equilibrium control $\hat{u}$ is characterized by the relation
\begin{equation*}
	H(t,\hat{x}_t,t,\hat{x}_t,\hat{u}_t,p^t_t,q^t_t)\leq H(t,\hat{x}_t,t,\hat{x}_t,v,p^t_t,q^t_t)
\end{equation*}
for any $t\in[0,T)$ and $v\in U$. Here, the function $H$ is the Hamiltonian defined by
\begin{equation*}
	H(t,\xi,s,x,u,p,q):=b(s,x,u)p+\sigma(s,x)q+f(t,\xi,s,x,u)
\end{equation*}
for $0\leq t\leq s\leq T$ and $\xi,x,p,q\in\mathbb{R}$; for each $t\in[0,T]$, $\left(p^t,q^t\right)=\left(p^t_s,q^t_s\right)_{s\in[t,T]}$ is the solution of the corresponding (first-order) adjoint equation, namely, the BSDE
\begin{equation*}
\begin{cases}
	dp^t_s=-\partial_xH(t,\hat{x}_t,s,\hat{x}_s,\hat{u}_s,p^t_s,q^t_s)\,ds+q^t_s\,dW_s,\ s\in[t,T],\\
	p^t_T=\partial_x g(t,\hat{x}_t,\hat{x}_T)+\partial_{\bar{x}}h(t,\hat{x}_t,\mathbb{E}_t[\hat{x}_T]),
\end{cases}
\end{equation*}
where $\partial_xH$ and $\partial_xg$ are the partial derivatives of $H$ and $g$ with respect to the $x$-variables (the fourth variable of $H$ and the third variable of $g$, respectively) and $\partial_{\bar{x}}h$ is the partial derivative of $h$ with respect to the $\bar{x}$-variable (the third variable of $h$).
\par
If the function $U\ni u\mapsto H(t,x,t,x,u,p,q)\in\mathbb{R}$ has a unique minimizer $\hat{u}(t,x,p)$ for each $t\in[0,T]$ and $x,p,q\in\mathbb{R}$, which is independent of $q$ since the volatility $\sigma$ is uncontrolled in this case, and $\hat{u}(t,x,p)$ satisfies an appropriate regularity condition, then the equilibrium control is characterized (at least formally) by $\hat{u}_t=\hat{u}(t,\hat{x}_t,p^t_t)$, $t\in[0,T]$, where the adapted processes $(\hat{x}_t,p^t_t),\ t\in[0,T]$, satisfy the following system:
\begin{equation*}
\begin{cases}
	d\hat{x}_s=b(s,\hat{x}_s,\hat{u}(s,\hat{x}_s,p^s_s))\,ds+\sigma(s,\hat{x}_s)\,dW_s,\ s\in[0,T],\\
	dp^t_s=-\partial_xH(t,\hat{x}_t,s,\hat{x}_s,\hat{u}(s,\hat{x}_s,p^s_s),p^t_s,q^t_s)\,ds+q^t_s\,dW_s,\ s\in[t,T],\\
	\hat{x}_0=x_0,\ p^t_T=\partial_x g(t,\hat{x}_t,\hat{x}_T)+\partial_{\bar{x}}h(t,\hat{x}_t,\mathbb{E}_t[\hat{x}_T]),\ t\in[0,T].
\end{cases}
\end{equation*}
This system is a special case of (\ref{flow of FBSDE}). We see that if we can solve this system, then we can construct the open-loop equilibrium control by using the equilibrium solution. In this paper, we investigate its solvability in a more general setting.
\par
Although characterizations of equilibrium controls by flows of forward-backward SDEs have been suggested in some papers, there are only a few studies about solvability of the equations. Hu, Jin and Zhou~\cite{a_Hu-12,a_Hu-17} studied linear-quadratic time-inconsistent stochastic control problems. They derived a flow of affine forward-backward SDEs with random coefficients characterizing the equilibrium control and solved it by using Riccati-like equations only when the state is one-dimensional and all the coefficients are deterministic. Djehiche and Huang~\cite{a_Djehiche-Huang_16} studied time-inconsistent mean-field stochastic control problems and derived a flow of forward-backward SDEs, while their models are assumed to be deterministic and solvability of the equations were not discussed.
\par
In this paper, in contrast to the above-mentioned papers, we investigate a flow of forward-backward SDEs with general and random coefficients and solve it by using a contraction mapping argument when the time-interval is sufficiently small. This is our first contribution. Furthermore, we introduce a discretized equilibrium solution which is a discretized version of the concept of equilibrium solutions and show that the discretized equilibrium solutions approximate the original one. In the case that $F$ is independent of $(X_t,\mathbb{E}_t[X_s])$ and $G$ is independent of $(X_t,\mathbb{E}_t[X_T])$, the discretized equilibrium solutions are written as adapted solutions of classical forward-backward SDEs, which have been analyzed in many papers; see the textbook \cite{b_Ma-Yong_99}. Thus our approximation result reveals that, in a special case, the equilibrium solution of the flow of forward-backward SDEs (\ref{flow of FBSDE}) is approximated by adapted solutions of the corresponding classical forward-backward SDEs. This type of approximation result is new and this is the second contribution of this paper. We hope that our approximation results provide a new insight for investigating a flow of forward-backward SDEs. On the other hand, the system (\ref{flow of FBSDE}) implicitly assume that the diagonal term $(Y^t_t)_{t\in[0,T]}$ is not only adapted but also progressively measurable. However, the progressive measurability of the diagonal term is not clear since for each $t\in[0,T]$ the random variable $Y^t_t$ derives from different BSDEs. We rigorously prove that, under natural assumptions, for given adapted processes $X$ and $\mathcal{Y}$ the diagonal term $(Y^t_t)_{t\in[0,T]}$ which derives from the BSDEs parametrized by $t\in[0,T]$:
\begin{equation*}
\begin{cases}
		dY^t_s=-F(t,s,X_t,\mathbb{E}_t[X_s],X_s,\mathcal{Y}_s,Y^t_s,Z^t_s)\,ds+Z^t_s\,dW_s,\ s\in[t,T],\\
		Y^t_T=G(t,X_t,\mathbb{E}_t[X_T],X_T),
\end{cases}
\end{equation*}
has a progressively measurable version; see Lemma \ref{measurability lemma}. This is an additional contribution of this paper.
\par
The system (\ref{flow of FBSDE}) is also regarded as a generalization of backward stochastic Volterra integral equations (BSVIEs) that were introduced by Lin~\cite{a_Lin_02} and studied by Yong~\cite{a_Yong_06}, Shi and Wang~\cite{a_Shi-Wang_12}, Shi, Wang and Yong~\cite{a_Shi-Wang-Yong_15}, Wang and Zhang~\cite{a_Wang-Zhang_17}, Wang and Yong~\cite{a_Wang-Yong_19}, Wang~\cite{a_Wang_19}, among others. Indeed, if the coefficients $F$ and $G$ are independent of $X$, the second line of (\ref{flow of FBSDE}) with the terminal condition $G(t)$ becomes an extended BSVIE (see \cite{a_Wang_19}) for the processes $\{(Y^t_s,Z^t_s)$; $0\leq t\leq s\leq T\}$, of the following form:
\begin{equation*}
	Y^t_s=G(t)+\int^T_sF(t,u,Y^u_u,Y^t_u,Z^t_u)\,du-\int^T_sZ^t_u\,dW_u,\ 0\leq t\leq s\leq T.
\end{equation*}
Thus, the flow of forward-backward SDEs (\ref{flow of FBSDE}) can be regarded as a fully coupled system consisting of an SDE and an extended BSVIE. We remark on this matter in Section~\ref{section remark}.
\par
Our paper is organized as follows: In Section~\ref{section FFBSDE}, we state the notations and prove the small-time solvability and stability estimates for the system (\ref{flow of FBSDE}). In Section~\ref{section approximation}, we introduce discretized equilibrium solutions and show an approximation result. Lastly, we provide some remarks and future problems in Section~\ref{section remark}.


\section{A flow of forward-backward SDEs}\label{section FFBSDE}


\subsection{Notations}\label{subsection notation}

\paragraph{}
\ \,In this subsection, we summarize the notations we use throughout the paper.
\par
$W=(W_t)_{t\geq0}$ is a $d$-dimensional Brownian motion defined on a probability space $(\Omega,\mathcal{F},\mathbb{P})$ and $\mathbb{F}=(\mathcal{F}_t)_{t\geq0}$ is the $\mathbb{P}$-augmentation of the filtration generated by $W$. We sometimes omit the dependency of $\omega$. For a set $A$, we denote by $\1_A$ the indicator function of $A$. $\text{Leb}$ denotes the Lebesgue measure on an interval. $\mathbb{E}_t[\cdot]:=\mathbb{E}[\cdot|\mathcal{F}_t]$ denotes the conditional expectation given $\mathcal{F}_t$. For $T>0$, we define $\Delta[0,T]:=\{(t,s)|0\leq t\leq s\leq T\}$.
\par
For $0\leq t\leq T<\infty$ and $\mathbb{H}=\mathbb{R}^n,\mathbb{R}^m$, etc., we define
\begin{align*}
	&L^2_{\mathcal{F}_t}(\Omega;\mathbb{H}):=\left\{\chi\relmiddle|
		\begin{aligned}
		&\chi\ \text{is}\ \mathbb{H}\text{-valued,}\ \mathcal{F}_t\text{-measurable}\\
		&\text{and satisfies}\ \mathbb{E}[|\chi|^2]<\infty
		\end{aligned}
	\right\},\displaybreak[1]\\
	&L^\infty_\mathbb{F}(t,T;L^2(\Omega;\mathbb{H})):=\left\{\chi=(\chi_s)_{s\in[t,T]}\relmiddle|
		\begin{aligned}
		&\chi\ \text{is}\ \mathbb{H}\text{-valued,}\ \mathbb{F}\text{-progressively measurable}\\
		&\text{and satisfies}\ \esssup\mathbb{E}[|\chi_s|^2]<\infty
		\end{aligned}
	\right\},\displaybreak[1]\\
	&L^2_\mathbb{F}(t,T;\mathbb{H}):=\left\{\chi=(\chi_s)_{s\in[t,T]}\relmiddle|
		\begin{aligned}
		&\chi\ \text{is}\ \mathbb{H}\text{-valued,}\ \mathbb{F}\text{-progressively measurable}\\
		&\text{and satisfies}\ \mathbb{E}\left[\int^T_t|\chi_s|^2\,ds\right]<\infty
		\end{aligned}
	\right\},\displaybreak[1]\\
	&L^2_\mathbb{F}(\Omega;C([t,T];\mathbb{H})):=\left\{\chi=(\chi_s)_{s\in[t,T]}\relmiddle|
		\begin{aligned}
		&\chi\ \text{is}\ \mathbb{H}\text{-valued,}\ \mathbb{F}\text{-adapted, continuous}\\
		&\text{and satisfies}\ \mathbb{E}\left[\sup_{s\in[t,T]}|\chi_s|^2\right]<\infty
		\end{aligned}
	\right\},\displaybreak[1]\\
	&L^2_\mathbb{F}(\Delta[0,T];\mathbb{H}):=\left\{(\chi^t_s)_{(t,s)\in\Delta[0,T]}\relmiddle|
		\begin{aligned}
		&(\chi^t_s)_{s\in[t,T]}\in L^2_\mathbb{F}(t,T;\mathbb{H}),\ \forall\,t\in[0,T],\\
		&\sup_{t\in[0,T]}\mathbb{E}\left[\int^T_t|\chi^t_s|^2\,ds\right]<\infty
		\end{aligned}
	\right\},\displaybreak[1]\ \shortintertext{and}
	&L^2_\mathbb{F}(\Omega;C(\Delta[0,T];\mathbb{H})):=\left\{(\chi^t_s)_{(t,s)\in\Delta[0,T]}\relmiddle|
		\begin{aligned}
		&(\chi^t_s)_{s\in[t,T]}\in L^2_\mathbb{F}(\Omega;C([t,T];\mathbb{H})),\ \forall\,t\in[0,T],\\
		&\sup_{t\in[0,T]}\mathbb{E}\left[\sup_{s\in[t,T]}|\chi^t_s|^2\right]<\infty
		\end{aligned}
	\right\}.
\end{align*}


\subsection{Small-time solvability and stability estimates}\label{subsection small-time solvability}

\paragraph{}
\ \,For $T>0$, we consider the system (\ref{flow of FBSDE}). We call this system a \emph{flow of forward-backward stochastic differential equations} and we use the notation $\text{FFBSDE}(T)$, where $T>0$ represents the terminal time of the system and the term ``FFBSDE'' stands for a ``Flow of Forward-Backward Stochastic Differential Equations''. The system (\ref{flow of FBSDE}) consists of a single forward SDE for $X$ and a continuum of BSDEs for $(Y^t,Z^t)$, $t\in[0,T]$, that are coupled via the ``diagonal term'' $Y^s_s$.
\par
We impose the following assumptions on the coefficients.

\begin{assum}
\begin{description}
\item[(A1)]
$x\in\mathbb{R}^n$. The mappings
\begin{align*}
	&[0,T]\times\mathbb{R}^n\times\mathbb{R}^m\times\Omega\ni(s,x,\eta,\omega)\mapsto B(s,x,\eta,\omega)\in\mathbb{R}^n\ \shortintertext{and}
	&[0,T]\times\mathbb{R}^n\times\mathbb{R}^m\times\Omega\ni(s,x,\eta,\omega)\mapsto\Sigma(s,x,\eta,\omega)\in\mathbb{R}^{n\times d}
\end{align*}
are $\left(\mathcal{F}_s\right)_{s\in[0,T]}$-progressively measurable. Moreover, for each $t\in[0,T]$, the mapping
\begin{align*}
	&[t,T]\times\mathbb{R}^n\times\mathbb{R}^n\times\mathbb{R}^n\times\mathbb{R}^m\times\mathbb{R}^m\times\mathbb{R}^{m\times d}\times\Omega\ni(s,\xi,\bar{x},x,\eta,y,z,\omega)\\
	&\hspace{2cm}\mapsto F(t,s,\xi,\bar{x},x,\eta,y,z,\omega)\in\mathbb{R}^m
\end{align*}
is $\left(\mathcal{F}_s\right)_{s\in[t,T]}$-progressively measurable and the mapping
\begin{equation*}
	\mathbb{R}^n\times\mathbb{R}^n\times\mathbb{R}^n\times\Omega\ni(\xi,\bar{x},x,\omega)\mapsto G(t,\xi,\bar{x},x,\omega)\in\mathbb{R}^m
\end{equation*}
is $\mathcal{B}(\mathbb{R}^n)\otimes\mathcal{B}(\mathbb{R}^n)\otimes\mathcal{B}(\mathbb{R}^n)\otimes\mathcal{F}_T$-measurable.
\end{description}
\begin{description}
\item[(A2)]
\begin{align*}
	&R:=\mathbb{E}\left[\left(\int^T_0|B_0(s)|\,ds\right)^2+\int^T_0|\Sigma_0(s)|^2\,ds\right]\\
	&\hspace{2cm}+\sup_{t\in[0,T]}\mathbb{E}\left[\left(\int^T_t|F_0(t,s)|\,ds\right)^2+|G_0(t)|^2\right]<\infty,
\end{align*}
where, for example, $B_0(s):=B(s,0,0)$, $F_0(t,s):=F(t,s,0,0,0,0,0,0)$, etc.
\item[(A3)]
There exists a constant $L>0$ such that
\begin{align*}
	&|B(s,x,\eta,\omega)-B(s,x',\eta',\omega)|\leq L(|x-x'|+|\eta-\eta'|),\\
	&|\Sigma(s,x,\eta,\omega)-\Sigma(s,x',\eta',\omega)|\leq L(|x-x'|+|\eta-\eta'|),\\
	&|F(t,s,\xi,\bar{x},x,\eta,y,z,\omega)-F(t,s,\xi',\bar{x}',x',\eta',y',z',\omega)|\\
	&\hspace{1cm}\leq L(|\xi-\xi'|+|\bar{x}-\bar{x}'|+|x-x'|+|\eta-\eta'|+|y-y'|+|z-z'|), \shortintertext{and}
	&|G(t,\xi,\bar{x},x,\omega)-G(t,\xi',\bar{x}',x',\omega)|\leq L(|\xi-\xi'|+|\bar{x}-\bar{x}'|+|x-x'|),
\end{align*}
for any $(t,s,\omega)\in\Delta[0,T]\times\Omega$, $\xi,\xi',\bar{x},\bar{x}',x,x'\in\mathbb{R}^n$, $\eta,\eta',y,y'\in\mathbb{R}^m$, and $z,z'\in\mathbb{R}^{m\times d}$.
\item[(A4)]
There exists an increasing function $\rho\colon[0,\infty)\to[0,\infty)$ with $\lim_{t\downarrow0}\rho(t)=\rho(0)=0$ such that
\begin{align*}
	&|F(t,s,\xi,\bar{x},x,\eta,y,z,\omega)-F(t',s,\xi,\bar{x},x,\eta,y,z,\omega)|+|G(t,\xi,\bar{x},x,\omega)-G(t',\xi,\bar{x},x,\omega)|\\
	&\leq\rho(|t-t'|)(1+|\xi|+|x|+|\bar{x}|+|\eta|+|y|+|z|)
\end{align*}
for any $\omega\in\Omega$, $s\in[0,T]$, $t,t'\in[0,s]$, $\xi,x,\bar{x}\in\mathbb{R}^n$, $\eta,y\in\mathbb{R}^m$ and $z\in\mathbb{R}^{m\times d}$.
\end{description}
\end{assum}

\par
In the sequel, the $\xi$-variable (resp.\ $\bar{x}$-variable) represents the $X_t$-dependency (resp.\ $\mathbb{E}_t[X_s]$-dependency) for the coefficients $F$ and $G$.
\par
At first, we define the concept of equilibrium solutions of FFBSDE$(T)$. The name ``equilibrium solutions'' comes from ``equilibrium controls'' for time-inconsistent stochastic control problems that are motivations to consider flows of forward-backward SDEs.

\begin{defi}\label{def equilibrium solution}
For each $T>0$, we call a triplet $(X,Y,Z)$ an \emph{equilibrium solution} of $\text{FFBSDE}(T)$ if $(X,Y,Z)\in L^2_\mathbb{F}(\Omega;C([0,T];\mathbb{R}^n))\times L^2_\mathbb{F}(\Omega;C(\Delta[0,T];\mathbb{R}^m))\times L^2_\mathbb{F}(\Delta[0,T];\mathbb{R}^{m\times d})$, the process $(Y^s_s)_{s\in[0,T]}$ is progressively measurable, and they satisfy equations in (\ref{flow of FBSDE}) in the usual It\^{o} sense, for each $t\in[0,T]$.
\end{defi}

Before proving the small-time solvability, we shall make some remarks on the definition of equilibrium solutions of FFBSDE$(T)$.

\begin{rem}
For each $t\in[0,T]$, since $(Y^t_s)_{s\in[t,T]}$ is a continuous adapted process defined on $[t,T]$, the diagonal term $Y^t_t$ is well-defined and it is $\mathcal{F}_t$-measurable.
\end{rem}

\begin{rem}
Since the system (\ref{flow of FBSDE}) has a continuum of backward equations, we need to be careful for ``$\mathbb{P}$-a.s.''\,validity of equations and measurability of the adapted process $(Y^s_s)_{s\in[0,T]}$. Our definition of equilibrium solutions imposes that, for each $t\in[0,T]$, $(Y^t_s,Z^t_s)_{s\in[t,T]}$ solves the BSDE on $[t,T]$ $\mathbb{P}$-a.s., its null set being allowed to depend on $t\in[0,T]$, and $(Y^s_s)_{s\in[0,T]}$ is progressively measurable.
\end{rem}

In fact, the following lemma guarantees the progressive measurability of the diagonal term $(Y^s_s)_{s\in[0,T]}$.

\begin{lemm}\label{measurability lemma}
 Let $T>0$ be arbitrary and suppose that Assumptions~(A1)--(A4) hold. Assume that we are given adapted processes $X\in L^2_\mathbb{F}(\Omega;C([0,T];\mathbb{R}^n))$ and $\mathcal{Y}\in L^2_\mathbb{F}(0,T;\mathbb{R}^m)$. For each $t\in[0,T]$, let $(Y^t,Z^t)\in L^2_\mathbb{F}(\Omega;C([t,T];\mathbb{R}^m))\times L^2_\mathbb{F}(t,T;\mathbb{R}^{m\times d})$ be the unique (up to a null set) solution of the BSDE
\begin{equation}\label{measurability lemma BSDE}
\begin{cases}
		dY^t_s=-F(t,s,X_t,\mathbb{E}_t[X_s],X_s,\mathcal{Y}_s,Y^t_s,Z^t_s)\,ds+Z^t_s\,dW_s,\ s\in[t,T],\\
		Y^t_T=G(t,X_t,\mathbb{E}_t[X_T],X_T).
\end{cases}
\end{equation}
Then, the process $(Y^s_s)_{s\in[0,T]}$ has a progressively measurable version.
\end{lemm}

\begin{proof}
In this proof, $C>0$ represents a constant which is independent of $(t,s)\in\Delta[0,T]$ and allowed to vary from line to line. By the standard argument of BSDEs (see, e.g. \cite{b_Zhang_17}), for any given $X\in L^2_\mathbb{F}(\Omega;C([0,T];\mathbb{R}^n))$ and $\mathcal{Y}\in L^2_\mathbb{F}(0,T;\mathbb{R}^m)$ and each $t\in[0,T]$, BSDE (\ref{measurability lemma BSDE}) has a unique adapted solution $(Y^t,Z^t)\in L^2_\mathbb{F}(\Omega;C([t,T];\mathbb{R}^m))\times L^2_\mathbb{F}(t,T;\mathbb{R}^{m\times d})$ and it holds that
\begin{align}\label{measurability lemma estimate 1}
	\nonumber&\mathbb{E}\left[\sup_{s\in[t,T]}\left|Y^t_s\right|^2+\int^T_t\left|Z^t_s\right|^2\,ds\right]\\
	&\leq C\mathbb{E}\left[\left|G(t,X_t,\mathbb{E}_t[X_T],X_T)\right|^2+\left(\int^T_t|F(t,s,X_t,\mathbb{E}_t[X_s],X_s,\mathcal{Y}_s,0,0)|\,ds\right)^2\right]
\end{align}
for any $t\in[0,T]$, and
\begin{align}\label{measurability lemma estimate 2}
	\nonumber&\mathbb{E}\left[\sup_{r\in[s,T]}\left|Y^s_r-Y^t_r\right|^2\right]\\
	\nonumber&\leq C\mathbb{E}\Biggl[\left|G(s,X_s,\mathbb{E}_s[X_T],X_T)-G(t,X_t,\mathbb{E}_t[X_T],X_T)\right|^2\\
	&\hspace{1cm}+\left(\int^T_s|F(s,r,X_s,\mathbb{E}_s[X_r],X_r,\mathcal{Y}_r,Y^s_r,Z^s_r)-F(t,r,X_t,\mathbb{E}_t[X_r],X_r,\mathcal{Y}_r,Y^s_r,Z^s_r)|\,dr\right)^2\Biggr]
\end{align}
for any $(t,s)\in\Delta[0,T]$. By Assumptions (A1)--(A2) and (\ref{measurability lemma estimate 1}), we see that
\begin{align}\label{measurability lemma estimate 3}
	\nonumber\sup_{t\in[0,T]}\mathbb{E}\left[\sup_{s\in[t,T]}\left|Y^t_s\right|^2+\int^T_t\left|Z^t_s\right|^2\,ds\right]&\leq C\left(R+\mathbb{E}\left[\sup_{s\in[0,T]}\left|X_s\right|^2+\int^T_0\left|\mathcal{Y}_s\right|^2\,ds\right]\right)\\
	&\leq C.
\end{align}
Assumptions (A3)--(A4) and inequalities (\ref{measurability lemma estimate 2})--(\ref{measurability lemma estimate 3}) yield that
\begin{align*}
	&\mathbb{E}\left[\left|Y^s_s-Y^t_s\right|^2\right]\\
	&\leq C\mathbb{E}\Biggl[\left|G(s,X_s,\mathbb{E}_s[X_T],X_T)-G(t,X_t,\mathbb{E}_t[X_T],X_T)\right|^2\\
	\displaybreak[1]
	&\hspace{1cm}+\left(\int^T_s|F(s,r,X_s,\mathbb{E}_s[X_r],X_r,\mathcal{Y}_r,Y^s_r,Z^s_r)-F(t,r,X_t,\mathbb{E}_t[X_r],X_r,\mathcal{Y}_r,Y^s_r,Z^s_r)|\,dr\right)^2\Biggr]\\
	&\leq C\left\{\mathbb{E}\Biggl[\left|G(s,X_s,\mathbb{E}_s[X_T],X_T)-G(t,X_s,\mathbb{E}_s[X_T],X_T)\right|^2\right.\\
	&\left.\hspace{1cm}+\left(\int^T_s|F(s,r,X_s,\mathbb{E}_s[X_r],X_r,\mathcal{Y}_r,Y^s_r,Z^s_r)-F(t,r,X_s,\mathbb{E}_s[X_r],X_r,\mathcal{Y}_r,Y^s_r,Z^s_r)|\,ds\right)^2\Biggr]\right.\\
	&\left.\hspace{1cm}+\mathbb{E}\Biggl[\left|G(t,X_s,\mathbb{E}_s[X_T],X_T)-G(t,X_t,\mathbb{E}_t[X_T],X_T)\right|^2\right.\\
	&\left.\hspace{1cm}+\left(\int^T_s|F(t,r,X_s,\mathbb{E}_s[X_r],X_r,\mathcal{Y}_r,Y^s_r,Z^s_r)-F(t,r,X_t,\mathbb{E}_t[X_r],X_r,\mathcal{Y}_r,Y^s_r,Z^s_r)|\,ds\right)^2\Biggr]\right\}\\
	\displaybreak[1]
	&\leq C\left(\rho(s-t)^2+\mathbb{E}[\left|X_s-X_t\right|^2]+\mathbb{E}\left[\left|\mathbb{E}_s[X_T]-\mathbb{E}_t[X_T]\right|^2\right]+\int^T_s\mathbb{E}\left[\left|\mathbb{E}_s[X_r]-\mathbb{E}_t[X_r]\right|^2\right]\,dr\right).
\end{align*}
Note that, for each $r\in[0,T]$, there exists a unique adapted process $\chi^r\in L^2_\mathbb{F}(0,r;\mathbb{R}^n)$ such that $X_r=\mathbb{E}[X_r]+\int^r_0\chi^r_u\,dW_u$, and hence
\begin{equation*}
	\mathbb{E}\left[\left|\mathbb{E}_s[X_T]-\mathbb{E}_t[X_T]\right|^2\right]=\int^s_t\mathbb{E}[|\chi^T_u|^2]\,du
\end{equation*}
and
\begin{align*}
	\int^T_s\mathbb{E}\left[\left|\mathbb{E}_s[X_r]-\mathbb{E}_t[X_r]\right|^2\right]\,dr&=\int^T_s\int^s_t\mathbb{E}[|\chi^r_u|^2]\,dudr=\int^s_t\int^T_s\mathbb{E}[|\chi^r_u|^2]\,drdu\\
	&\leq\int^s_t\int^T_u\mathbb{E}[|\chi^r_u|^2]\,drdu
\end{align*}
for $(t,s)\in\Delta[0,T]$. Thus, we obtain
\begin{equation}\label{measurability lemma estimate 4}
	\mathbb{E}\left[\left|Y^s_s-Y^t_s\right|^2\right]\leq C\left(\rho(s-t)^2+\mathbb{E}[|X_s-X_t|^2]+\int^s_t\mathbb{E}[|\chi^T_u|^2]\,du+\int^s_t\int^T_u\mathbb{E}[|\chi^r_u|^2]\,drdu\right).
\end{equation}
Note that, since $X\in L^2_\mathbb{F}(\Omega;C([0,T];\mathbb{R}^n))$, the mapping $[0,T]\ni s\mapsto X_s\in L^2_{\mathcal{F}_T}(\Omega;\mathbb{R}^n)$ is continuous, and hence it is uniformly continuous. Furthermore, since $\chi^T\in L^2_\mathbb{F}(0,T;\mathbb{R}^n)$ and
\begin{align*}
	\int^T_0\int^T_u\mathbb{E}[|\chi^r_u|^2]\,drdu=\int^T_0\int^r_0\mathbb{E}[|\chi^r_u|^2]\,dudr=\int^T_0(\mathbb{E}[|X_r|^2]-|\mathbb{E}[X_r]|^2)\,dr<\infty,
\end{align*}
we see that the functions $[0,T]\ni s\mapsto \int^s_0\mathbb{E}[|\chi^T_u|^2]\,du$ and $[0,T]\ni s\mapsto\int^s_0\int^T_u\mathbb{E}[|\chi^r_u|^2]\,drdu$ are continuous, and hence they are uniformly continuous. Thus, the right hand side of (\ref{measurability lemma estimate 4}) tends to zero as $t\uparrow s$ uniformly in $s\in(0,T]$. Hence, for each $k\in\mathbb{N}$, there exists a finite partition $\Pi^k=\{t^k_l|l=0,1,\dots,N_k\}$, $0=t^k_0<t^k_1<\dots<t^k_{N_k}=T$, such that
\begin{equation*}
	\mathbb{P}\left\{\left|Y^s_s-Y^{t^k_{l-1}}_s\right|\geq\frac{1}{2^k}\right\}\leq\frac{1}{2^k}\ \text{if}\ s\in\left[t^k_{l-1},t^k_l\right),\ l=1,\dots,N_k.
\end{equation*}
Define
\begin{equation*}
	\eta^k_s(\omega):=
	\begin{cases}
		Y^{t^k_{l-1}}_s(\omega)\ &\text{if}\ s\in\left[t^k_{l-1},t^k_l\right),\ l=1,\dots,N_k,\\
		Y^T_T(\omega)\ &\text{if}\ s=T,
	\end{cases}
\end{equation*}
for each $(s,\omega)\in[0,T]\times\Omega$ and $k\in\mathbb{N}$. Then, the processes $\eta^k=\left(\eta^k_s\right)_{s\in[0,T]}$ are progressively measurable and so is the set $A:=\left\{(s,\omega)\in[0,T]\times\Omega\relmiddle|\exists\ \lim_{k\to\infty}\eta^k_s(\omega)\right\}$. Hence, the process $\eta=\left(\eta_s\right)_{s\in[0,T]}$ defined by
\begin{equation*}
	\eta_s(\omega):=
	\begin{cases}
		\lim_{k\to\infty}\eta^k_s(\omega)\ &\text{if}\ (s,\omega)\in A,\\
		0\ &\text{if}\ (s,\omega)\notin A,
	\end{cases}
\end{equation*}
is also progressively measurable. Fix an arbitrary $s\in[0,T]$. Then, the Borel--Cantelli lemma yields that, for $\mathbb{P}$-a.e.\,$\omega\in\Omega$, $\left|Y^s_s(\omega)-\eta^k_s(\omega)\right|\leq2^{-k}$ holds for any sufficiently large $k\in\mathbb{N}$, and hence $\left(\eta_s\right)_{s\in[0,T]}$ is a version of $(Y^s_s)_{s\in[0,T]}$. 
\end{proof}

Now we prove the small-time solvability of a flow of forward-backward SDEs.

\begin{theo}\label{theorem small-time solvability}
Suppose that Assumptions (A1)--(A4) hold. Then there exists a constant $T_0>0$ which depends only on the Lipschitz constant $L$ such that when $T\leq T_0$ there exists a unique equilibrium solution $(X,Y,Z)\in L^2_\mathbb{F}(\Omega;C([0,T];\mathbb{R}^n))\times L^2_\mathbb{F}(\Omega;C(\Delta[0,T];\mathbb{R}^m))\times L^2_\mathbb{F}(\Delta[0,T];\mathbb{R}^{m\times d})$ of FFBSDE$(T)$.\\
\end{theo}

\begin{proof}
Let $0<T\leq T_0$ with $T_0>0$ being determined later. In this proof, we denote by $C$ a positive constant which depends only on $L$ and is allowed to vary from line to line. For each $\mathcal{Y}\in L^\infty_\mathbb{F}(0,T;L^2(\Omega;\mathbb{R}^m))$, define $\Phi(\mathcal{Y})\in L^\infty_\mathbb{F}(0,T;L^2(\Omega;\mathbb{R}^m))$ by the following procedure:
\begin{enumerate}
\renewcommand{\labelenumi}{(\roman{enumi})}
\item
Define $X\in L^2_\mathbb{F}(\Omega;C([0,T];\mathbb{R}^n))$ as the unique solution of the SDE
\begin{equation*}
\begin{cases}
	dX_s=B(s,X_s,\mathcal{Y}_s)\,ds+\Sigma(s,X_s,\mathcal{Y}_s)\,dW_s,\ s\in[0,T],\\
	X_0=x.
\end{cases}
\end{equation*}
\item
For each $t\in[0,T]$, define $(Y^t,Z^t)\in L^2_\mathbb{F}(\Omega;C([t,T];\mathbb{R}^m))\times L^2_\mathbb{F}(t,T;\mathbb{R}^{m\times d})$ as the unique solution of the BSDE
\begin{equation*}
\begin{cases}
	dY^t_s=-F(t,s,X_t,\mathbb{E}_t[X_s],X_s,\mathcal{Y}_s,Y^t_s,Z^t_s)\,ds+Z^t_s\,dW_s,\ s\in[t,T],\\
	Y^t_T=G(t,X_t,\mathbb{E}_t[X_T],X_T).
\end{cases}
\end{equation*}
\item
Define $\Phi(\mathcal{Y})$ by $\Phi(\mathcal{Y}):=\tilde{\mathcal{Y}}$, where $\tilde{\mathcal{Y}}=(\tilde{\mathcal{Y}}_s)_{s\in[0,T]}$ is a progressively measurable version of the process $(Y^s_s)_{s\in[0,T]}$; see Lemma~\ref{measurability lemma}.
\end{enumerate}
By the estimate (\ref{measurability lemma estimate 3}), the family of processes $\{Y^t\}_{t\in[0,T]}$ defined in (\rnum{2}) satisfies
\begin{equation*}
	\sup_{t\in[0,T]}\mathbb{E}\left[\sup_{s\in[t,T]}\left|Y^t_s\right|^2\,ds\right]\leq C\left(R+\mathbb{E}\left[\sup_{s\in[0,T]}|X_s|^2+\int^T_0|\mathcal{Y}_s|^2\,ds\right]\right)<\infty,
\end{equation*}
Hence the process $\tilde{\mathcal{Y}}$ defined in (\rnum{3}) is in $L^\infty_\mathbb{F}(0,T;L^2(\Omega;\mathbb{R}^m))$. Note that $L^\infty_\mathbb{F}(0,T;L^2(\Omega;\mathbb{R}^m))$ is a Banach space with respect to the norm
\begin{equation*}
	\|\mathcal{Y}\|_\infty:=\esssup\mathbb{E}\left[|\mathcal{Y}_s|^2\right]^{1/2}.
\end{equation*}
In order to prove well-posedness of $\text{FFBSDE}(T)$, it suffices to prove that $\Phi$ is a contraction mapping on $L^\infty_\mathbb{F}(0,T;L^2(\Omega;\mathbb{R}^m))$.
\par
For given inputs $\mathcal{Y}^1,\mathcal{Y}^2\in L^\infty_\mathbb{F}(0,T;L^2(\Omega;\mathbb{R}^m))$, define $X^1,X^2\in L^2_\mathbb{F}(\Omega;C([0,T];\mathbb{R}^n))$, $(Y^{1,t},Z^{1,t}),(Y^{2,t},Z^{2,t})\in L^2_\mathbb{F}(\Omega;C([t,T];\mathbb{R}^m))\times L^2_\mathbb{F}(t,T;\mathbb{R}^{m\times d})$, $t\in[0,T]$, and $\tilde{\mathcal{Y}}^1,\tilde{\mathcal{Y}}^2\in L^\infty_\mathbb{F}(0,T;L^2(\Omega;\mathbb{R}^m))$ by the above procedure (\rnum{1})--(\rnum{3}). 
By using the BDG inequality and Gronwall's inequality, we can easily show that
\begin{equation*}
	\mathbb{E}\left[\sup_{s\in[0,T]}\left|X^1_s-X^2_s\right|^2\right]\leq C\mathbb{E}\left[\int^T_0\left|\mathcal{Y}^1_s-\mathcal{Y}^2_s\right|^2\,ds\right].
\end{equation*}
Then, by the standard estimate of BSDEs, we see that
\begin{align*}
	&\sup_{t\in[0,T]}\mathbb{E}\left[\sup_{s\in[t,T]}\left|Y^{1,t}_s-Y^{2,t}_s\right|^2\right]\\
	&\leq C\sup_{t\in[0,T]}\mathbb{E}\Biggl[\left|G(t,X^1_t,\mathbb{E}_t[X^1_T],X^1_T)-G(t,X^2_t,\mathbb{E}_t[X^2_T],X^2_T)\right|^2\\
	\displaybreak[1]
	&\hspace{0.3cm}+\left(\int^T_t|F(t,s,X^1_t,\mathbb{E}_t[X^1_s],X^1_s,\mathcal{Y}^1_s,Y^{1,t}_s,Z^{1,t}_s)-F(t,s,X^2_t,\mathbb{E}_t[X^2_s],X^2_s,\mathcal{Y}^2_s,Y^{1,t}_s,Z^{1,t}_s)|\,ds\right)^2\Biggr]\\
	&\leq C\mathbb{E}\left[\sup_{s\in[0,T]}\left|X^1_s-X^2_s\right|^2+\int^T_0\left|\mathcal{Y}^1_s-\mathcal{Y}^2_s\right|^2\,ds\right]\\
	&\leq C\mathbb{E}\left[\int^T_0\left|\mathcal{Y}^1_s-\mathcal{Y}^2_s\right|^2\,ds\right].
\end{align*}
In particular, we have
\begin{equation*}
	\|\tilde{\mathcal{Y}}^1-\tilde{\mathcal{Y}}^2\|_\infty^2\leq C\mathbb{E}\left[\int^T_0\left|\mathcal{Y}^1_s-\mathcal{Y}^2_s\right|^2\,ds\right]\leq CT\|\mathcal{Y}^1-\mathcal{Y}^2\|^2_\infty.
\end{equation*}
Thus, if $T_0=T_0(L)>0$ is sufficiently small and $0<T\leq T_0$, then $\Phi$ is a contraction mapping on $L^\infty_\mathbb{F}(0,T;L^2(\Omega;\mathbb{R}^m))$. This completes the proof.
\end{proof}

Next, we prove a stability estimate of equilibrium solutions of flows of forward-backward SDEs.

\begin{theo}\label{theorem stability}
Let $\left(x^1,B^1,\Sigma^1,F^1,G^1\right)$ and $\left(x^2,B^2,\Sigma^2,F^2,G^2\right)$ be coefficients satisfying Assumptions~(A1)--(A4) with constants $\left(R_1,L_1\right)$ and $\left(R_2,L_2\right)$, respectively. Assume that there exists a constant $\tilde{T}>0$ such that, for any $0<T\leq \tilde{T}$ and $i=1,2$, there exists a unique equilibrium solution $(X^i,Y^i,Z^i)$ of $\text{FFBSDE}(T)$ with the coefficients $\left(x^i,B^i,\Sigma^i,F^i,G^i\right)$. Then, there exist constants $0<T_1\leq\tilde{T}$ and $C>0$ that depend only on $L_1$ such that, for each $0<T\leq T_1$, it holds that
\begin{align}\label{theorem stability estimate}
	\nonumber&\mathbb{E}\left[\sup_{s\in[0,T]}\left|X^1_s-X^2_s\right|^2\right]+\sup_{t\in[0,T]}\mathbb{E}\left[\sup_{s\in[t,T]}\left|Y^{1,t}_s-Y^{2,t}_s\right|^2+\int^T_t\left|Z^{1,t}_s-Z^{2,t}_s\right|^2\,ds\right]\\
	\nonumber&\leq C\Biggl(\left|x^1-x^2\right|^2+\mathbb{E}\left[\left(\int^T_0|\delta B(s,X^2_s,Y^{2,s}_s)|\,ds\right)^2+\int^T_0\left|\delta\Sigma(s,X^2_s,Y^{2,s}_s)\right|^2\,ds\right]\\
	\nonumber&\hspace{1cm}+\sup_{t\in[0,T]}\mathbb{E}\Biggl[\left|\delta G(t,X^2_t,\mathbb{E}_t[X^2_T],X^2_T)\right|^2\\
	&\hspace{3cm}+\left(\int^T_t|\delta F(t,s,X^2_t,\mathbb{E}_t[X^2_s],X^2_s,Y^{2,s}_s,Y^{2,t}_s,Z^{2,t}_s)|\,ds\right)^2\Biggr]\Biggr),
\end{align}
where $\delta\Phi:=\Phi^1-\Phi^2$ for $\Phi=B,\Sigma,F,G$.
In particular, it holds that
\begin{equation}\label{theorem estimate}
	\mathbb{E}\left[\sup_{s\in[0,T]}\left|X^1_s\right|^2\right]+\sup_{t\in[0,T]}\mathbb{E}\left[\sup_{s\in[t,T]}\left|Y^{1,t}_s\right|^2+\int^T_t\left|Z^{1,t}_s\right|^2\,ds\right]\leq C\left(\left|x^1\right|^2+R_1\right).
\end{equation}
\end{theo}

\begin{proof}
In this proof, we denote by $C$ a positive constant which depends only on $L_1$ and is allowed to vary from line to line. Let $\mathcal{Y}^1_s:=Y^{1,s}_s$ and $\mathcal{Y}^2_s:=Y^{2,s}_s$ for $s\in[0,T]$. Moreover, in this proof we use the notation
\begin{equation*}
	\mathcal{X}^i_{t,s}=(X^i_t,\mathbb{E}_t[X^i_s],X^i_s),\ (t,s)\in\Delta[0,T],\ i=1,2.
\end{equation*}
Then for $i=1,2$ and for each $t\in[0,T]$, $(Y^{i,t},Z^{i,t})\in L^2_\mathbb{F}(\Omega;C([t,T];\mathbb{R}^m))\times L^2_\mathbb{F}(t,T;\mathbb{R}^{m\times d})$ solves the following standard BSDE:
\begin{equation*}
	\begin{cases}
		dY^{i,t}_s=-F^i(t,s,\mathcal{X}^i_{t,s},\mathcal{Y}^i_s,Y^{i,t}_s,Z^{i,t}_s)\,ds+Z^{i,t}_s\,dW_s,\ s\in[t,T],\\
		Y^{i,t}_T=G(t,\mathcal{X}^i_{t,T}).
	\end{cases}
\end{equation*}
By the stability estimate of solutions of standard BSDEs, we have, for each $t\in[0,T]$,
\begin{align}\label{theorem stability estimate 1}
	\nonumber&\mathbb{E}\left[\sup_{s\in[t,T]}\left|Y^{1,t}_s-Y^{2,t}_s\right|^2+\int^T_t\left|Z^{1,t}_s-Z^{2,t}_s\right|^2\,ds\right]\\
	\nonumber&\leq C\mathbb{E}\Biggl[\left|G^1(t,\mathcal{X}^1_{t,T})-G^2(t,\mathcal{X}^2_{t,T})\right|^2\\
	\nonumber&\hspace{2cm}+\Biggl(\int^T_t|F^1(t,s,\mathcal{X}^1_{t,s},\mathcal{Y}^1_s,Y^{2,t}_s,Z^{2,t}_s)-F^2(t,s,\mathcal{X}^2_{t,s},\mathcal{Y}^2_s,Y^{2,t}_s,Z^{2,t}_s)|\,ds\Biggr)^2\Biggr]\displaybreak[1]\\
	\nonumber&\leq C\mathbb{E}\Biggl[\sup_{s\in[t,T]}\left|X^1_s-X^2_s\right|^2+\int^T_t\left|\mathcal{Y}^1_s-\mathcal{Y}^2_s\right|^2\,ds\\
	&\hspace{2cm}+\left|\delta G(t,\mathcal{X}^2_{t,T})\right|^2+\left(\int^T_t|\delta F(t,s,\mathcal{X}^2_{t,s},\mathcal{Y}^2_s,Y^{2,t}_s,Z^{2,t}_s)|\,ds\right)^2\Biggr].
\end{align}
In particular, we have
\begin{align*}
	&\mathbb{E}\left[\left|\mathcal{Y}^1_t-\mathcal{Y}^2_t\right|^2\right]\\
	&\leq C\Biggl(\int^T_t\mathbb{E}\left[\left|\mathcal{Y}^1_s-\mathcal{Y}^2_s\right|^2\right]\,ds+\mathbb{E}\left[\sup_{s\in[0,T]}\left|X^1_s-X^2_s\right|^2\right]\\
	&\hspace{1.5cm}+\sup_{t\in[0,T]}\mathbb{E}\left[\left|\delta G(t,\mathcal{X}^2_{t,T})\right|^2+\left(\int^T_t|\delta F(t,s,\mathcal{X}^2_{t,s},\mathcal{Y}^2_s,Y^{2,t}_s,Z^{2,t}_s)|\,ds\right)^2\right]\Biggr)
\end{align*}
for any $t\in[0,T]$. Then Gronwall's inequality yields that
\begin{align}\label{theorem stability estimate 2}
	\nonumber&\sup_{s\in[0,T]}\mathbb{E}\left[\left|\mathcal{Y}^1_s-\mathcal{Y}^2_s\right|^2\right]\\
	\nonumber&\leq C\Biggl(\mathbb{E}\left[\sup_{s\in[0,T]}\left|X^1_s-X^2_s\right|^2\right]\\
	&\hspace{1.5cm}+\sup_{t\in[0,T]}\mathbb{E}\left[\left|\delta G(t,\mathcal{X}^2_{t,T})\right|^2+\left(\int^T_t|\delta F(t,s\mathcal{X}^2_{t,s},\mathcal{Y}^2_s,Y^{2,t}_s,Z^{2,t}_s)|\,ds\right)^2\right]\Biggr).
\end{align}
On the other hand, by the stability estimate of solutions of SDEs, we have
\begin{align}\label{theorem stability estimate 3}
	\nonumber&\mathbb{E}\left[\sup_{s\in[0,T]}\left|X^1_s-X^2_s\right|^2\right]\\
	\nonumber&\leq C\Biggl(\left|x^1-x^2\right|^2+\mathbb{E}\left[\left(\int^T_0|B(s,X^2_s,\mathcal{Y}^1_s)-B(s,X^2_s,\mathcal{Y}^2_s)|\,ds\right)^2\right]\\
	\nonumber&\hspace{1cm}+\mathbb{E}\left[\int^T_0\left|\Sigma(s,X^2_s,\mathcal{Y}^1_s)-\Sigma(s,X^2_s,\mathcal{Y}^2_s)\right|^2\,ds\right]\Biggr)\displaybreak[1]\\
	\nonumber&\leq C\Biggl(\left|x^1-x^2\right|^2+\mathbb{E}\left[\left(\int^T_0|\delta B(s,X^2_s,\mathcal{Y}^2_s)|\,ds\right)^2+\int^T_0\left|\delta\Sigma(s,X^2_s,\mathcal{Y}^2_s)\right|^2\,ds\right]\\
	&\hspace{1cm}+\mathbb{E}\left[\int^T_0\left|\mathcal{Y}^1_s-\mathcal{Y}^2_s\right|^2\,ds\right]\Biggr).
\end{align}
Combining (\ref{theorem stability estimate 2}) and (\ref{theorem stability estimate 3}), we obtain
\begin{align*}
	&\sup_{s\in[0,T]}\mathbb{E}\left[\left|\mathcal{Y}^1_s-\mathcal{Y}^2_s\right|^2\right]\\
	&\leq C\Biggl(\int^T_0\mathbb{E}\left[\left|\mathcal{Y}^1_s-\mathcal{Y}^2_s\right|^2\right]\,ds+\left|x^1-x^2\right|^2\\
	&\hspace{1.5cm}+\mathbb{E}\left[\left(\int^T_0|\delta B(s,X^2_s,\mathcal{Y}^2_s)|\,ds\right)^2+\int^T_0\left|\delta\Sigma(s,X^2_s,\mathcal{Y}^2_s)\right|^2\,ds\right]\\
	&\hspace{1.5cm}+\sup_{t\in[0,T]}\mathbb{E}\left[\left|\delta G(t,\mathcal{X}^2_{t,T})\right|^2+\left(\int^T_t|\delta F(t,s,\mathcal{X}^2_{t,s},\mathcal{Y}^2_s,Y^{2,t}_s,Z^{2,t}_s)|\,ds\right)^2\right]\Biggr).
\end{align*}
Since $\int^T_0\mathbb{E}\left[\left|\mathcal{Y}^1_s-\mathcal{Y}^2_s\right|^2\right]\,ds\leq T\sup_{s\in[0,T]}\mathbb{E}\left[\left|\mathcal{Y}^1_s-\mathcal{Y}^2_s\right|^2\right]$, if $T_1=T_1(L_1)>0$ is small and $T\leq T_1$, it holds that
\begin{align*}
	&\sup_{s\in[0,T]}\mathbb{E}\left[\left|\mathcal{Y}^1_s-\mathcal{Y}^2_s\right|^2\right]\\
	&\leq C\Biggl(\left|x^1-x^2\right|^2+\mathbb{E}\left[\left(\int^T_0|\delta B(s,X^2_s,\mathcal{Y}^2_s)|\,ds\right)^2+\int^T_0\left|\delta\Sigma(s,X^2_s,\mathcal{Y}^2_s)\right|^2\,ds\right]\\
	&\hspace{1.5cm}+\sup_{t\in[0,T]}\mathbb{E}\left[\left|\delta G(t,\mathcal{X}^2_{t,T})\right|^2+\left(\int^T_t|\delta F(t,s,\mathcal{X}^2_{t,s},\mathcal{Y}^2_s,Y^{2,t}_s,Z^{2,t}_s)|\,ds\right)^2\right]\Biggr).
\end{align*}
Inserting this inequality to (\ref{theorem stability estimate 1}) and (\ref{theorem stability estimate 3}), we obtain (\ref{theorem stability estimate}). By considering $x^2=0$ and $B^2,\Sigma^2,F^2,G^2=0$, we easily see that the estimate (\ref{theorem estimate}) holds.
\end{proof}


\section{Approximation of the equilibrium solution}\label{section approximation}

\paragraph{}
\ \,In this section, we introduce \emph{discretized equilibrium solutions} of FFBSDE$(T)$ which approximate the original equilibrium solution in an appropriate sense. In a special case, discretized equilibrium solutions turn out to be adapted solutions of classical forward-backward SDEs.
\par
Note that the system (\ref{flow of FBSDE}) can be rewritten as the equations
\begin{equation}\label{flow of FBSDE 2}
\begin{cases}
	dX_s=B(s,X_s,\mathcal{Y}_s)\,ds+\Sigma(s,X_s,\mathcal{Y}_s)\,dW_s,\ s\in[0,T],\\
	dY^t_s=-F(t,s,X_t,\mathbb{E}_t[X_s],X_s,\mathcal{Y}_s,Y^t_s,Z^t_s)\,ds+Z^t_s\,dW_s,\ s\in[t,T],\\
	X_0=x,\ Y^t_T=G(t,X_t,\mathbb{E}_t[X_T],X_T),\ t\in[0,T],
\end{cases}
\end{equation}
with the condition
\begin{equation}\label{diagonal condition}
	\mathcal{Y}_s=Y^s_s,\ \text{Leb}\otimes\mathbb{P}\text{-a.e.}\,(s,\omega)\in[0,T]\times\Omega.
\end{equation}
The system (\ref{flow of FBSDE 2}) consists of a continuum of equations with parameter $t\in[0,T]$ and all the equations are coupled via (\ref{diagonal condition}). Now we consider a discretized version of the condition (\ref{diagonal condition}) so that the system (\ref{flow of FBSDE 2}) reduces to a coupled system consisting of only finitely many equations. To do so, let $\mathcal{P}[0,T]$ be the set of finite partitions $\Pi$ of the interval $[0,T]$; $\Pi=\{t_k|k=0,1,\dots,N\}$, $0=t_0<t_1<\dots<t_N=T$. $\|\Pi\|:=\max_{k=1,\dots,N}(t_k-t_{k-1})$ denotes the mesh size of $\Pi$. For each $\Pi\in\mathcal{P}[0,T]$, we consider the following system:
\begin{equation}\label{discrete flow of FBSDE}
\begin{cases}
	dX^\Pi_s=B(s,X^\Pi_s,\mathcal{Y}^\Pi_s)\,ds+\Sigma(s,X^\Pi_s,\mathcal{Y}^\Pi_s)\,dW_s,\ s\in[0,T],\\
		dY^{\Pi,k}_s=-F(t_{k-1},s,X^\Pi_{t_{k-1}},\mathbb{E}_{t_{k-1}}[X^\Pi_s],X^\Pi_s,\mathcal{Y}^\Pi_s,Y^{\Pi,k}_s,Z^{\Pi,k}_s)\,ds+Z^{\Pi,k}_s\,dW_s,\ s\in[t_{k-1},T],\\
		X^\Pi_0=x,\ Y^{\Pi,k}_T=G(t_{k-1},X^\Pi_{t_{k-1}},\mathbb{E}_{t_{k-1}}[X^\Pi_T],X^\Pi_T),\ 
k=1,\dots,N,\\
	\mathcal{Y}^\Pi_s=\sum^{N-1}_{j=1}Y^{\Pi,j}_s\1_{[t_{j-1},t_j)}(s)+Y^{\Pi,N}_s\1_{[t_{N-1},t_N]}(s),\ s\in[0,T].
\end{cases}
\end{equation}

\begin{defi}\label{def Pi-equilibrium solution}
For each $T>0$ and $\Pi=\{t_k|k=0,1,\dots,N\}\in\mathcal{P}[0,T]$, we call a triplet $(X^\Pi,Y^\Pi,Z^\Pi)$ a \emph{$\Pi$-equilibrium solution} of  FFBSDE$(T)$ if $X^\Pi\in L^2_\mathbb{F}(\Omega;C([0,T];\mathbb{R}^n))$, $(Y^\Pi,Z^\Pi)=\{(Y^{\Pi,k},Z^{\Pi,k})\}_{k=1,\dots,N}$ with $(Y^{\Pi,k},Z^{\Pi,k})\in L^2_\mathbb{F}(\Omega;C([t_{k-1},T];\mathbb{R}^m))\times L^2_\mathbb{F}(t_{k-1},T;\mathbb{R}^{m\times d})$, $k=1,\dots,N$, and they satisfy equations in (\ref{discrete flow of FBSDE}) in the usual It\^{o} sense, for each $k=1,\dots,N$.
\end{defi}

By the same argument as Theorem \ref{theorem small-time solvability}, we obtain the following proposition.

\begin{prop}\label{theorem small-time solvability Pi}
Suppose that Assumptions (A1)--(A4) hold and let $T_0=T_0(L)>0$ be the constant appearing in Theorem \ref{theorem small-time solvability}.  Then for any $T\leq T_0$ and any $\Pi\in\mathcal{P}[0,T]$ there exists a unique $\Pi$-equilibrium solution $(X^\Pi,Y^\Pi,Z^\Pi)$ of FFBSDE$(T)$.
\end{prop}

The next theorem says that the discretized equilibrium solutions of a flow of forward-backward SDEs approximate the original equilibrium solution in an appropriate sense.

\begin{theo}\label{theorem approximation}
Suppose that Assumptions (A1)--(A4) hold and let $\tilde{T}>0$ be a constant for which the assertions of Theorems \ref{theorem small-time solvability}, \ref{theorem stability} and Proposition \ref{theorem small-time solvability Pi} hold. For each $T\leq\tilde{T}$ and $\Pi\in\mathcal{P}[0,T]$, denote by $(X,Y,Z)$ (resp.\ $(X^\Pi,Y^\Pi,Z^\Pi)$) the equilibrium solution (resp.\ $\Pi$-equilibrium solution) of FFBSDE$(T)$ and let $\mathcal{Y}_s=Y^s_s$, $\mathcal{Y}^\Pi_s=\sum^{N-1}_{k=1}Y^{\Pi,k}_s\1_{[t_{k-1},t_k)}(s)+Y^{\Pi,N}_s\1_{[t_{N-1},t_N]}(s)$, $s\in[0,T]$. Then there exists a constant $T_2\leq \tilde{T}$ which depends only on $L$ such that, for any $T\leq T_2$, it holds that
\begin{equation}\label{theorem Pi limit}
	\lim_{\|\Pi\|\to0}\left(\mathbb{E}\left[\sup_{s\in[0,T]}\left|X^\Pi_s-X_s\right|^2\right]+\sup_{s\in[0,T]}\mathbb{E}\left[\left|\mathcal{Y}^\Pi_s-\mathcal{Y}_s\right|^2\right]\right)=0.
\end{equation}
Furthermore, If $F$ and $G$ are independent of the $(\xi,\bar{x})$-variables, then there exists a constant $C>0$ which depends only on $L$ such that, for any $T\leq T_2$ and any $\Pi\in\mathcal{P}[0,T]$, it holds that
\begin{equation}\label{theorem Pi estimate}
	\mathbb{E}\left[\sup_{s\in[0,T]}\left|X^\Pi_s-X_s\right|^2\right]+\sup_{s\in[0,T]}\mathbb{E}\left[\left|\mathcal{Y}^\Pi_s-\mathcal{Y}_s\right|^2\right]\leq C\left(R+|x|^2\right)\rho(\|\Pi\|)^2.
\end{equation}
\end{theo} 

\begin{proof}
As before, we denote by $C>0$ a constant which depends only on $L$ and is allowed to vary from line to line. Let $0<T\leq\tilde{T}$ and $\Pi\in\mathcal{P}[0,T]$ with $\Pi=\{t_k|k=0,1,\dots,N\}$, $0=t_0<t_1<\dots<t_N=T$. In this proof, we use the notations
\begin{equation*}
	\mathcal{X}_{t,s}=(X_t,\mathbb{E}_t[X_s],X_s),\ \mathcal{X}^\Pi_{t,s}=(X^\Pi_t,\mathbb{E}_t[X^\Pi_s],X^\Pi_s),\ (t,s)\in\Delta[0,T].
\end{equation*}
Fix an arbitrary $t\in[0,T]$ and let $k$ be the number such that $t\in[t_{k-1},t_k)$; when $t=T$, let $k=N$. Then the standard estimate of solutions of BSDEs yields that
\begin{align*}
	&\mathbb{E}\left[\sup_{s\in[t,T]}\left|Y^t_s-Y^{\Pi,k}_s\right|^2\right]\\
	&\leq C\mathbb{E}\Biggl[\left|G(t,\mathcal{X}_{t,T})-G(t_{k-1},\mathcal{X}^\Pi_{t_{k-1},T})\right|^2\\
	&\hspace{2cm}+\left(\int^T_t|F(t,s,\mathcal{X}_{t,s},\mathcal{Y}_s,Y^t_s,Z^t_s)-F(t_{k-1},s,\mathcal{X}^\Pi_{t_{k-1},s},\mathcal{Y}^\Pi_s,Y^t_s,Z^t_s)|\,ds\right)^2\Biggr]\displaybreak[1]\\
	&\leq C\Biggl\{\mathbb{E}\Biggl[\left|G(t,\mathcal{X}_{t,T})-G(t_{k-1},\mathcal{X}_{t,T})\right|^2\\
	&\hspace{2cm}+\Biggl(\int^T_t|F(t,s,\mathcal{X}_{t,s},\mathcal{Y}_s,Y^t_s,Z^t_s)-F(t_{k-1},s,\mathcal{X}_{t,s},\mathcal{Y}_s,Y^t_s,Z^t_s)|\,ds\Biggr)^2\Biggr]\\
	&\hspace{1cm}+\mathbb{E}\Biggl[\left|G(t_{k-1},\mathcal{X}_{t,T})-G(t_{k-1},\mathcal{X}_{t_{k-1},T})\right|^2\\
	&\hspace{2cm}+\Biggl(\int^T_t|F(t_{k-1},s,\mathcal{X}_{t,s},\mathcal{Y}_s,Y^t_s,Z^t_s)-F(t_{k-1},s,\mathcal{X}_{t_{k-1},s},\mathcal{Y}_s,Y^t_s,Z^t_s)|\,ds\Biggr)^2\Biggr]\\
	&\hspace{1cm}+\mathbb{E}\Biggl[\left|G(t_{k-1},\mathcal{X}_{t_{k-1},T})-G(t_{k-1},\mathcal{X}^\Pi_{t_{k-1},T})\right|^2\\
	&\hspace{2cm}+\Biggl(\int^T_t|F(t_{k-1},s,\mathcal{X}_{t_{k-1},s},\mathcal{Y}_s,Y^t_s,Z^t_s)-F(t_{k-1},s,\mathcal{X}^\Pi_{t_{k-1},s},\mathcal{Y}^\Pi_s,Y^t_s,Z^t_s)|\,ds\Biggr)^2\Biggr]\Biggr\}\\
	&=:C\left(I_1(t;\Pi)+I_2(t;\Pi)+I_3(t;\Pi)\right).
\end{align*}
Since $\mathcal{Y}_t=Y^t_t$ and $\mathcal{Y}^{\Pi}_t=Y^{\Pi,k}_t$, we obtain
\begin{equation}\label{I_123}
	\mathbb{E}\left[\left|\mathcal{Y}_t-\mathcal{Y}^\Pi_t\right|^2\right]\leq C(I_1(t;\Pi)+I_2(t;\Pi)+I_3(t,\Pi)).
\end{equation}
Assumption~(A3) yields that
\begin{align}\label{I_3 estimate}
	\nonumber I_3(t;\Pi)&\leq C\mathbb{E}\left[\sup_{s\in[t,T]}\left|X_s-X^\Pi_s\right|^2+\int^T_t\left|\mathcal{Y}_s-\mathcal{Y}^\Pi_s\right|^2\,ds\right]\\
	\nonumber&\leq C\mathbb{E}\left[\int^T_0\left|\mathcal{Y}_s-\mathcal{Y}^\Pi_s\right|^2\,ds\right]\\
	&\leq CT\sup_{t\in[0,T]}\mathbb{E}\left[\left|\mathcal{Y}_t-\mathcal{Y}^\Pi_t\right|^2\right],
\end{align}
where, in the second inequality, we used the estimate
\begin{equation}\label{x y estimate}
	\mathbb{E}\left[\sup_{s\in[0,T]}\left|X_s-X^\Pi_s\right|^2\right]\leq C\mathbb{E}\left[\int^T_0\left|\mathcal{Y}_s-\mathcal{Y}^\Pi_s\right|^2\,ds\right],
\end{equation}
which can be easily shown by using the standard estimate of solutions of SDEs. By (\ref{I_123}) and (\ref{I_3 estimate}), if $T_2=T_2(L)>0$ is sufficiently small and $T\leq T_2$, then it holds that
\begin{equation}\label{I_12}
	\sup_{t\in[0,T]}\mathbb{E}\left[\left|\mathcal{Y}_t-\mathcal{Y}^\Pi_t\right|^2\right]\leq C\left(\sup_{t\in[0,T]}I_1(t;\Pi)+\sup_{t\in[0,T]}I_2(t;\Pi)\right).
\end{equation}
By Assumption~(A4) and the estimate (\ref{theorem estimate}) (noting that $\mathcal{Y}_s=Y^s_s,\ s\in[0,T]$), the expectation $I_1(t;\Pi)$ (the difference with respect to the $t$-variable) can be estimated as follows:
\begin{align}\label{I_1 estimate}
	\nonumber I_1(t;\Pi)&\leq C\rho(t-t_{k-1})^2\left(1+\mathbb{E}\left[\sup_{s\in[t,T]}\left|X_s\right|^2+\int^T_t\left(\left|\mathcal{Y}_s\right|^2+\left|Y^t_s\right|^2+\left|Z^t_s\right|^2\right)\,ds\right]\right)\\
	\nonumber&\leq C\left(R+|x|^2\right)\rho(t-t_{k-1})^2\\
	&\leq C\left(R+|x|^2\right)\rho(\|\Pi\|)^2,
\end{align}
in particular,
\begin{equation}\label{I_1 limit}
	\lim_{\|\Pi\|\to0}\sup_{t\in[0,T]}I_1(t,\Pi)=0.
\end{equation}
By Assumption~(A3), the expectation $I_2(t;\Pi)$ (the difference with respect to the $(\xi,\bar{x})$-variables) can be estimated as follows:
\begin{equation*}
	I_2(t;\Pi)\leq C\mathbb{E}\left[\left|X_t-X_{t_{k-1}}\right|^2+\left|\mathbb{E}_t[X_T]-\mathbb{E}_{t_{k-1}}[X_T]\right|^2+\int^T_t\left|\mathbb{E}_t[X_s]-\mathbb{E}_{t_{k-1}}[X_s]\right|^2\,ds\right].
\end{equation*}
Note that $k$ depends on $t\in[0,T]$ and $t-t_{k-1}\leq\|\Pi\|$. By the same argument as in the proof of Lemma \ref{measurability lemma}, we see that the right hand side of the above inequality tends to zero as the mesh size $\|\Pi\|$ tends to zero uniformly in $t\in[0,T]$. Hence, we have
\begin{equation}\label{I_2 limit}
	\lim_{\|\Pi\|\to0}\sup_{t\in[0,T]}I_2(t,\Pi)=0.
\end{equation}
Thus, if $T\leq T_2$, By (\ref{I_12}), (\ref{I_1 limit}) and (\ref{I_2 limit}), we obtain
\begin{equation}\label{Pi limit}
	\lim_{\|\Pi\|\to0}\sup_{t\in[0,T]}\mathbb{E}\left[\left|\mathcal{Y}_t-\mathcal{Y}^\Pi_t\right|^2\right]=0.
\end{equation}
If $F$ and $G$ do not depend on the $(\xi,\bar{x})$-variables, then $I_2(t;\Pi)=0$, $t\in[0,T]$, and hence (\ref{I_12}) and $(\ref{I_1 estimate})$ yield that
\begin{equation}\label{Pi estimate}
	\sup_{t\in[0,T]}\mathbb{E}\left[\left|\mathcal{Y}_t-\mathcal{Y}^\Pi_t\right|^2\right]\leq C\left(R+|x|^2\right)\rho(\|\Pi\|)^2.
\end{equation}
From (\ref{x y estimate}) and (\ref{Pi limit})--(\ref{Pi estimate}), we obtain (\ref{theorem Pi limit}) and (\ref{theorem Pi estimate}). This completes the proof.
\end{proof}

\begin{rem}
If $F$ and $G$ are independent of the $(\xi,\bar{x})$-variables, then for each $\Pi\in\mathcal{P}[0,T]$ the system (\ref{discrete flow of FBSDE}) reduces to
\begin{equation}\label{discrete flow of FBSDE 2}
\begin{cases}
	dX^\Pi_s=B(s,X^\Pi_s,\mathcal{Y}^\Pi_s)\,ds+\Sigma(s,X^\Pi_s,\mathcal{Y}^\Pi_s)\,dW_s,\ s\in[0,T],\\
		dY^{\Pi,k}_s=-F(t_{k-1},s,X^\Pi_s,\mathcal{Y}^\Pi_s,Y^{\Pi,k}_s,Z^{\Pi,k}_s)\,ds+Z^{\Pi,k}_s\,dW_s,\ s\in[t_{k-1},T],\\
		X^\Pi_0=x,\ Y^{\Pi,k}_T=G(t_{k-1},X^\Pi_T),\ 
k=1,\dots,N,\\
	\mathcal{Y}^\Pi_s=\sum^{N-1}_{j=1}Y^{\Pi,j}_s\1_{[t_{j-1},t_j)}(s)+Y^{\Pi,N}_s\1_{[t_{N-1},t_N]}(s),\ s\in[0,T].
\end{cases}
\end{equation}
Without loss of generality we extend $F$ to the mapping $F\colon[0,T]^2\times\mathbb{R}^n\times\mathbb{R}^m\times\mathbb{R}^m\times\mathbb{R}^{m\times d}\times\Omega\to\mathbb{R}^m$ by letting $F(t,s,\dots):=0$ for $0\leq s<t\leq T$. Then the system (\ref{discrete flow of FBSDE 2}) can be seen as a classical forward-backward SDEs:
\begin{equation}\label{classical FBSDE}
	\begin{cases}
	dX^\Pi_s=B^\Pi(s,X^\Pi_s,\bm{Y}^\Pi_s)\,ds+\Sigma^\Pi(s,X^\Pi_s,\bm{Y}^\Pi_s)\,dW_s,\\
	d\bm{Y}^\Pi_s=-\bm{F}^\Pi(s,X^\Pi_s,\bm{Y}^\Pi_s,\bm{Z}^\Pi_s)\,ds+\bm{Z}^\Pi_s\,dW_s,\ s\in[0,T],\\
	X^\Pi_0=x,\ \bm{Y}^\Pi_T=\bm{G}^\Pi(X^\Pi_T),
	\end{cases}
\end{equation}
where, $X^\Pi\in L^2_\mathbb{F}(\Omega;C([0,T];\mathbb{R}^n))$, $\bm{Y}^\Pi=\left(
    \begin{array}{c}
      Y^{\Pi,1}\\
      \vdots \\
      Y^{\Pi,N}
    \end{array}
  \right)\in L^2_\mathbb{F}(\Omega;C([0,T];\mathbb{R}^{mN}))$, and $\bm{Z}^\Pi=\left(
    \begin{array}{c}
      Z^{\Pi,1}\\
      \vdots \\
      Z^{\Pi,N}
    \end{array}
  \right)\in L^2_\mathbb{F}(0,T;\mathbb{R}^{(mN)\times d})$; the coefficients $B^\Pi,\Sigma^\Pi,\bm{F}^\Pi,\bm{G}^\Pi$ are defined by
\begin{align*}
	&B^\Pi(s,x,\bm{y}):=B(s,x,\phi^\Pi(s,\bm{y})),\ \Sigma^\Pi(s,x,\bm{y}):=\Sigma(s,x,\phi^\Pi(s,\bm{y})),\\
	&\bm{F}^\Pi(s,x,\bm{y},\bm{z}):=\left(
    \begin{array}{c}
      F(t_0,s,x,\phi^\Pi(s,\bm{y}),y^1,z^1) \\
      F(t_1,s,x,\phi^\Pi(s,\bm{y}),y^2,z^2)\\
      \vdots \\
      F(t_{N-1},s,x,\phi^\Pi(s,\bm{y}),y^N,z^N)
    \end{array}
  \right),
	\ \bm{G}^\Pi(x):=\left(
    \begin{array}{c}
      G(t_0,x)\\
      G(t_1,x)\\
      \vdots \\
      G(t_{N-1},x)
    \end{array}
  \right),\displaybreak[1]\ \shortintertext{with}
	&\phi^\Pi(s,\bm{y}):=\sum^{N-1}_{k=1}y^k\1_{[t_{k-1},t_k)}(s)+y^N\1_{[t_{N-1},t_N]}(s),
\end{align*}
for $s\in[0,T],\ x\in\mathbb{R}^n,\ \bm{y}=\left(
    \begin{array}{c}
      y^1\\
      \vdots \\
      y^N
    \end{array}
  \right)\in\mathbb{R}^{mN},\ \text{and}\ \bm{z}=\left(
    \begin{array}{c}
      z^1\\
      \vdots \\
      z^N
    \end{array}
  \right)\allowbreak\in\mathbb{R}^{(mN)\times d}$. In the above notations, it holds that $\mathcal{Y}^\Pi_s=\phi^\Pi(s,\bm{Y}^\Pi_s),\ s\in[0,T]$. Theorem~\ref{theorem approximation} says that, under our assumptions, the equilibrium solution of FFBSDE$(T)$ can be approximated by the adapted solutions of the classical forward-backward SDEs (\ref{classical FBSDE}) in the sense that the following estimate holds:
\begin{equation*}
	\mathbb{E}\left[\sup_{s\in[0,T]}|X^\Pi_s-X_s|^2\right]+\sup_{s\in[0,T]}\mathbb{E}\left[|\phi^\Pi(s,\bm{Y}^\Pi_s)-\mathcal{Y}_s|^2\right]\leq C\left(R+|x|^2\right)\rho(\|\Pi\|)^2
\end{equation*}
for each $\Pi\in\mathcal{P}[0,T]$.
\end{rem}


\section{Concluding remarks and future problems}\label{section remark}

\paragraph{}
\ \,We conclude this paper by discussing two future problems.
\par
The first problem is solvability of flows of forward-backward SDEs on arbitrary time-intervals $[0,T]$. This is a difficult problem since even in the case of classical forward-backward SDEs Lipschitz continuity of the coefficients is insufficient for well-posedness of the equation defined on an arbitrary time-interval; see the textbook \cite{b_Ma-Yong_99}. In classical forward-backward SDE theory, the so-called four-step scheme introduced by Ma, Protter and Yong \cite{a_Ma-Protter-Yong_94} is a good method to treat the case where the time-interval is arbitrary. This method is to decouple the backward and forward components of the equation by the so-called decoupling field which turns out to be a classical solution of a quasilinear PDE. Since there are two time variables $(t,s)$ in the case of flows of forward-backward SDEs, we have to generalize the concept of decoupling fields in order to take these variables into account. Indeed, Wang~\cite{a_Wang_19} recently introduced the extended BSVIE of the following form:
\begin{equation}\label{EBSVIE}
	\begin{cases}
		X_s=x+\int^s_0B(r,X_r)\,dr+\int^s_0\Sigma(r,X_r)\,dW_r,\\
		Y^t_s=G(t,X_T)+\int^T_sF(t,r,X_r,Y^r_r,Y^t_r,Z^t_r)\,dr-\int^T_sZ^t_r\,dW_r,\ (t,s)\in\Delta[0,T],
	\end{cases}
\end{equation}
(with $n=d$,) and investigated the connection between the above equation and the non-local quasilinear PDE system:
\begin{equation*}
	\begin{cases}
		\theta_s(t,s,x)+\theta_x(t,s,x)B(s,x)+\frac{1}{2}\Sigma(s,x)^\top\theta_{xx}(t,s,x)\Sigma(s,x)\\
		\hspace{0.3cm}+F(t,s,x,\theta(s,s,x),\theta(t,s,x),\theta_x(t,s,x)\Sigma(s,x))=0,\ (t,s,x)\in\Delta[0,T]\times\mathbb{R}^d,\\
		\theta(t,T,x)=G(t,x),\ (t,x)\in[0,T]\times\mathbb{R}^d.
	\end{cases}
\end{equation*}
The equation (\ref{EBSVIE}) is nothing but the decoupled version of FFBSDE$(T)$ with the coefficients $B$ and $\Sigma$ being independent of $Y^s_s$ and $F$ and $G$ being independent of $(X_t,\mathbb{E}_t[X_s])$. From the above literature, we can predict that, in the case that $F$ and $G$ are independent of $(X_t,\mathbb{E}_t[X_s])$, the generalization of the four-step scheme to a (coupled) flow of forward-backward SDEs (\ref{flow of FBSDE}) will reduce to the investigation of the non-local PDE system of the following form:
\begin{equation*}
	\begin{cases}
		\theta_s(t,s,x)+\theta_x(t,s,x)B(s,x,\theta(s,s,x))+\frac{1}{2}\Sigma(s,x,\theta(s,s,x))^\top\theta_{xx}(t,s,x)\Sigma(s,x,\theta(s,s,x))\\
		\hspace{0.3cm}+F(t,s,x,\theta(s,s,x),\theta(t,s,x),\theta_x(t,s,x)\Sigma(s,x,\theta(s,s,x)))=0,\ (t,s,x)\in\Delta[0,T]\times\mathbb{R}^d,\\
		\theta(t,T,x)=G(t,x),\ (t,x)\in[0,T]\times\mathbb{R}^d.
	\end{cases}
\end{equation*}
We expect that this generalization of the four-step scheme helps us to solve FFBSDE$(T)$ globally. This is our future problem.
\par
The second problem is a generalization to more intricately coupled systems. In order to treat the time-inconsistent stochastic control problems where the volatility of the state process is also controlled, we should consider more general forms of forward-backward systems, namely, the following form of flows:
\begin{equation}\label{general flow of FBSDE}
\begin{cases}
	dX_s=B(s,X_s,Y^s_s,Z^s_s)\,ds+\Sigma(s,X_s,Y^s_s,Z^s_s)\,dW_s,\ s\in[0,T],\\
	dY^t_s=-F(t,s,X_t,\mathbb{E}_t[X_s],X_s,Y^s_s,Z^s_s,Y^t_s,Z^t_s)\,ds+Z^t_s\,dW_s,\ s\in[t,T],\\
	X_0=x,\ Y^t_T=G(t,X_t,\mathbb{E}_t[X_T],X_T),\ t\in[0,T].
\end{cases}
\end{equation}
Unfortunately, our arguments in this paper are insufficient to treat this generalized system. A difficulty comes from the lack of regularity of $Z$. Indeed, for each $t\in[0,T]$, the process $(Z^t_s)_{\in[t,T]}$ is  in $L^2_\mathbb{F}(t,T;\mathbb{R}^{m\times d})$, not in $L^2_\mathbb{F}(\Omega;C([t,T];\mathbb{R}^{m\times d}))$, in general. So even the well-definedness of the diagonal term $Z^s_s$ is not clear. In order to investigate the system (\ref{general flow of FBSDE}), we have to estimate the term $Z$ in more detail, which is yet to be investigated.


\section*{Acknowledgments}
\paragraph{}
\ \,I would like to thank Professor Masanori Hino, who is my supervisor, Professor Ichiro Shigekawa, and Professor Jun Sekine for helpful discussions.
\par
This work was supported by JSPS KAKENHI Grant Number JP18J20973.
\\
\par
Conflict of Interest: The author declares that he has no conflict of interest.


\bibliography{reference}


\end{document}